\documentclass[]{interact}

\usepackage{epstopdf}% To incorporate .eps illustrations using PDFLaTeX, etc.
\usepackage[caption=false]{subfig}% Support for small, `sub' figures and tables
\usepackage{xcolor}
\usepackage{amsmath}
\usepackage{amsmath,cases}
\usepackage{mathtools}
\usepackage{xcolor}
\usepackage[numbers,sort&compress]{natbib}% Citation support using natbib.sty
\usepackage{mathtools}
\bibpunct[, ]{[}{]}{,}{n}{,}{,}% Citation support using natbib.sty
\renewcommand\bibfont{\fontsize{10}{12}\selectfont}% Bibliography support using natbib.sty
\makeatletter% @ becomes a letter
\def\NAT@def@citea{\def\@citea{\NAT@separator}}% Suppress spaces between citations using natbib.sty
\makeatother% @ becomes a symbol again

\theoremstyle{plain}% Theorem-like structures provided by amsthm.sty
\newtheorem{theorem}{Theorem}[section]
\newtheorem{lemma}[theorem]{Lemma}

\newtheorem{proposition}[theorem]{Proposition}

\theoremstyle{definition}
\newtheorem{definition}[theorem]{Definition}

\theoremstyle{remark}
\newtheorem{remark}{Remark}

\begin{document}

\articletype{Research Paper}% Specify the article type or omit as appropriate

\title{Three solutions for a double phase variable exponent Kirchhoff problem}

\author{\name{Mustafa Avci\thanks{CONTACT M.~Avci. Email:  mavci@athabascau.ca (primary) \& avcixmustafa@gmail.com}}
\affil{Faculty of Science and Technology, Applied Mathematics, Athabasca University, AB, Canada}}

\maketitle

\begin{abstract}
In this article, we study a double-phase variable-exponent Kirchhoff problem and show the existence of at least three solutions. The proposed model, as a generalization of the Kirchhoff equation, is interesting since it is driven by  a double-phase operator that governs anisotropic and heterogeneous diffusion associated with the energy functional, as well as encapsulating two different types of elliptic behavior within the same framework. To tackle the problem, we obtain regularity results for the corresponding energy functional, which makes the problem suitable for the application of a well-known critical point result by Bonanno and Marano. We introduce an $n$-dimensional vector inequality, not covered in the literature, which provides a key auxiliary tool for establishing essential regularity properties of the energy functional such as $C^1$-smoothness, the $(S_+)$-condition, and sequential weak lower semicontinuity.
\end{abstract}


\begin{amscode}
35A01; 35A15; 35B38; 35D30; 35J20; 35J60
\end{amscode}

\section{Introduction}\label{sec1}
In this article, we study the following two double phase variable exponent Kirchhoff problem

\begin{equation}\label{e1.1a}
\left\{\begin{array}{ll}
\displaystyle-\mathcal{M}\left(\int_\Omega
\frac{|\nabla u|^{p(x)}}{p(x)}+\mu(x)\frac{|\nabla u|^{q(x)}}{q(x)}
\right)&\\
\hspace{0.4cm} \times \hspace{0.1cm} \mathrm{div}\left(|\nabla u|^{p(x)-2}\nabla u+\mu(x)|\nabla u|^{q(x)-2}\nabla u\right)=\lambda f(x,u) \text{ in}\ \Omega,\\
\hspace{7.3cm} u=0  \quad \quad \quad \text{ on }\partial \Omega,\tag{$\mathcal{P_\lambda}$}
\end{array}
\right.
\end{equation}
where $\Omega$ is a bounded domain in $\mathbb{R}^N$ $(N\geq2)$ with Lipschitz boundary; $p,q\in C_+(\overline{\Omega })$; $f$ is a Carathéodory function; and $\lambda>0$ is a real parameter.\\

The problem (\ref{e1.1a}) indicates a generalization of the Kirchhoff equation \cite{kirchhoff1883mechanik}. Initially, Kirchhoff suggested a model given by the equation
\begin{equation*}
\rho \frac{\partial ^{2}u}{\partial t^{2}}-\left(\frac{P_{0}}{h}+\frac{E}{2l}\int_{0}^{l}\left\vert \frac{\partial u}{\partial x}\right\vert
^{2}dx\right) \frac{\partial ^{2}u}{\partial x^{2}}=0,
\end{equation*}
where $\rho $, $P_{0}$, $h$, $E$, $l$ are constants, which extends the classical D'Alambert's wave equation, by considering the effects of the changes in the length of the strings during the vibrations. However, since then, this model and its various perturbed versions, especially in the variable exponent setting, have been studied intensively by many authors. \\

Equations of the form (\ref{e1.1a}) appear in models involving materials with non-uniform (anisotropic) properties, fluid mechanics, image processing, and elasticity in non-homogeneous materials. Consequently, problem (\ref{e1.1a}) can be used to model various real-world phenomena, primarily due to the presence of the operator
\begin{equation}\label{e1.2a}
\mathrm{div}\left(|\nabla u|^{p(x)-2}\nabla u+\mu(x)|\nabla u|^{q(x)-2}\nabla u\right)
\end{equation}
 which governs anisotropic and heterogeneous diffusion associated with the energy functional
\begin{equation}\label{e1.2b}
u\to \int_{\Omega}\left(\frac{|\nabla u|^{p(x)}}{p(x)}+\mu(x)\frac{|\nabla u|^{q(x)}}{q(x)}\right)dx,\,\ u \in W_0^{1,\mathcal{H}}(\Omega)
\end{equation}
which is called a "double-phase'' operator. {This functional displays varying ellipticity depending on the regions where the weight function $\mu(\cdot)$ vanishes, thereby transitioning between two distinct elliptic phases. Zhikov \cite{zhikov1987averaging} was the first to study this type of functional with constant exponents, aiming to describe the behavior of strongly anisotropic materials. In the framework of elasticity theory, the function $\mu(\cdot)$ reflects the geometric structure of composites made from two distinct materials characterized by power-law hardening exponents $p(\cdot)$ and $q(\cdot)$.} Since then, numerous studies have explored this topic due to its wide applicability across various disciplines (e.g., \cite{baroni2015harnack,baroni2018regularity,colombo2015bounded,colombo2015regularity,marcellini1991regularity,marcellini1989regularity,galewski2024variational,bu2022p,lapa2015no,el2023existence,albalawi2022gradient}).
The interested reader may also refer to \cite{radulescu2019isotropic} for an overview of the isotropic and anisotropic double-phase problems.

The analysis of problem (\ref{e1.1a}) requires the framework of Musielak–Orlicz Sobolev spaces. In the context of fluid mechanics, Orlicz spaces provide a natural and more flexible framework than classical $L^p$ spaces for the modeling of materials whose rheological behavior deviates from standard power-law profiles. They accommodate more general growth conditions such as logarithmic corrections $t^p \log(1 + t)$ or exponential-type responses that frequently arise in the study of slow or fast diffusion, polymeric flows, and plasticity. These phenomena are often encountered in non-Newtonian and complex fluids, where the stress–strain relationship is nonlinear and varies with physical parameters like pressure, electric field, or temperature.
%MDPI: Please confirm if the italics should be retained. please check the whole text.%Author: italics removed here and throughout the paper. also m dashes "—" removed here.

We refer to the following related papers where double-phase variable-exponent Kirchhoff problems are studied. In \cite{zuo2024double}, the authors study a class of double-phase variable-exponent problems of the Kirchhoff type. Using the sub-supersolution method within an appropriate Musielak--Orlicz Sobolev space framework, they demonstrate the existence of at least one nonnegative solution. Moreover, by imposing an additional assumption on the nonlinearity, the authors employ variational arguments to establish the existence of a second nonnegative solution. In \cite{el2025existence}, the authors address a class of Kirchhoff-type problems in a double-phase setting with a small perturbation. The authors provide a new, less restrictive assumption than the (AR)-condition, which is a crucial tool in applying the Mountain Pass Theorem, under which the problem admits at least two weak solutions. The proof is based on variational arguments, utilizing the Mountain Pass Theorem with the Cerami condition. {In contrast to the works \cite{zuo2024double,el2025existence}, the principal novelty of this paper lies in the introduction of a new $n$-dimensional vector inequality, presented in Proposition \ref{Prop:2.2bc}. This inequality serves as a crucial auxiliary tool for establishing key regularity properties of the energy functional $\mathcal{K}$ and its derivative $\mathcal{K}^{\prime}$, including $C^1$-smoothness, the $(S_+)$-condition, and sequential weak lower semicontinuity.}

The paper is organized as follows. In Section \ref{sec2}, we first provide some background for the theory of variable Sobolev spaces $W_{0}^{1,p(x)}(\Omega)$ and the Musielak--Orlicz Sobolev space $W_0^{1,\mathcal{H}}(\Omega)$. In this section, we prove Proposition \ref{Prop:2.2bc}, which is the main originality of the paper, as well as the crucial auxiliary result  showing that the functional $\mathcal{K}$ is continuously G\^{a}teaux-differentiable, which is one of the main difficulties in the study of problem (\ref{e1.1a}). In Section \ref{sec3}, we  obtain another crucial auxiliary result, namely Lemmas \ref{Lem:4.2} and \ref{Lem:4.3}, where we show the required regularity assumptions of the corresponding functionals $\mathcal{I}_{\lambda}$, $\mathcal{K}$, and $\mathcal{J}$ of problem (\ref{e1.1a}). Then, we show that problem (\ref{e1.1a}) admits at least three distinct weak solutions by applying the well-known critical point result given by Bonanno and and Marano \cite{bonanno2010structure} (Theorem 3.6).

\section{Mathematical Background and Auxiliary Results}\label{sec2}

We start with some basic concepts of variable Lebesgue-Sobolev spaces. For more details, and the proof of the following propositions, we refer the reader to \cite{cruz2013variable,diening2011lebesgue,edmunds2000sobolev,fan2001spaces,radulescu2015partial}.\\
Define the set
\begin{equation*}
C_{+}\left( \overline{\Omega }\right) =\left\{h\in C(\overline{\Omega }):  h(x)>1 \text{ for all\ }x\in
\overline{\Omega }\right\} .
\end{equation*}
For $h\in C_{+}( \overline{\Omega }) $, we denote
\begin{equation*}
h^{-}:=\underset{x\in \overline{\Omega }}{\min }h(x) \leq h(x) \leq h^{+}:=\underset{x\in \overline{\Omega }}{\max}h(x) <\infty.
\end{equation*}
We also use the following notations ($a,b,t \in \mathbb{R}_+$):
\begin{equation*}
\min t^{h(x)}=t^{h_{-}}:=
\left\{ \begin{array}{ll}
t^{h^+},\quad & t< 1 \\
t^{h^-},\quad & t\geq 1
\end{array}
\right. \text{;  } \max t^{h(x)}=t^{h_{+}}=:
\left\{ \begin{array}{ll}
t^{h^-},\quad & t< 1  \\
t^{h^+},\quad & t\geq 1
\end{array}\right.
\end{equation*}
and
\begin{equation*}
t^{a\vee b}:=
\left\{ \begin{array}{ll}
t^{\min\{a,b\}},\quad & t< 1 \\
t^{\max\{a,b\}},\quad & t\geq 1
\end{array}
\right. \text{;  } t^{a\wedge b}:=
\left\{ \begin{array}{ll}
t^{\max\{a,b\}},\quad & t< 1 \\
t^{\min\{a,b\}},\quad & t\geq 1
\end{array}\right.
\end{equation*}
For any $h\in C_{+}\left( \overline{\Omega }\right) $, we define \textit{the
variable exponent Lebesgue space} by
\begin{equation*}
L^{h(x)}(\Omega) =\left\{ u\mid u:\Omega\rightarrow\mathbb{R}\text{ is measurable},\int_{\Omega }|u(x)|^{h(x) }dx<\infty \right\}.
\end{equation*}
Then, $L^{h(x)}(\Omega)$ endowed with the norm
\begin{equation*}
|u|_{h(x)}=\inf \left\{ \lambda>0:\int_{\Omega }\left\vert \frac{u(x)}{\lambda }
\right\vert ^{h(x)}dx\leq 1\right\} ,
\end{equation*}
becomes a Banach space.\\
The convex functional $\rho :L^{h(x) }(\Omega) \rightarrow\mathbb{R}$ defined by
\begin{equation*}
\rho(u) =\int_{\Omega }|u(x)|^{h(x)}dx,
\end{equation*}
is called modular on $L^{h(x) }(\Omega)$.

\begin{proposition}\label{Prop:2.2} If $u,u_{n}\in L^{h(x) }(\Omega)$, we have
\begin{itemize}
\item[$(i)$] $|u|_{h(x) }<1 ( =1;>1) \Leftrightarrow \rho(u) <1 (=1;>1);$
\item[$(ii)$] $|u|_{h( x) }>1 \implies |u|_{h(x)}^{h^{-}}\leq \rho(u) \leq |u|_{h( x) }^{h^{+}}$;\newline
$|u|_{h(x) }\leq1 \implies |u|_{h(x) }^{h^{+}}\leq \rho(u) \leq |u|_{h(x) }^{h^{-}};$
\item[$(iii)$] $\lim\limits_{n\rightarrow \infty }|u_{n}-u|_{h(x)}=0\Leftrightarrow \lim\limits_{n\rightarrow \infty }\rho (u_{n}-u)=0$.
\end{itemize}
\end{proposition}

\begin{proposition}\label{Prop:2.2bb}
Let $h_1(x)$ and $h_2(x)$ be measurable functions such that $h_1\in L^{\infty}(\Omega )$ and $1\leq h_1(x)h_2(x)\leq \infty$ for a.e. $x\in \Omega$. Let $u\in L^{h_2(x)}(\Omega ),~u\neq 0$. Then
\begin{itemize}
\item[$(i)$] $\left\vert u\right\vert _{h_1(x)h_2(x)}\leq 1\text{\ }\Longrightarrow
\left\vert u\right\vert _{h_1(x)h_2(x)}^{h_1^{+}}\leq \left\vert \left\vert
u\right\vert ^{h_1(x)}\right\vert _{h_2(x)}\leq \left\vert
u\right\vert _{h_1(x)h_2(x)}^{h^{-}}$
\item[$(ii)$] $\left\vert u\right\vert _{h_1(x)h_2(x)}> 1\ \Longrightarrow \left\vert
u\right\vert _{h_1(x)h_2(x)}^{h_1^{-}}\leq \left\vert \left\vert u\right\vert
^{h_1(x)}\right\vert _{h_2(x) }\leq \left\vert u\right\vert
_{h_1(x)h_2(x)}^{h_1^{+}}$
\item[$(iii)$] In particular, if $h_1(x)=h$ is constant then
\begin{equation*}
\left\vert \left\vert u\right\vert ^{h}\right\vert _{h_2(x)}=\left\vert
u\right\vert _{hh_2(x)}^{h}.
\end{equation*}
\end{itemize}
\end{proposition}
The variable exponent Sobolev space $W^{1,h(x)}( \Omega)$ is defined by
\begin{equation*}
W^{1,h(x)}(\Omega) =\{u\in L^{h(x) }(\Omega) : |\nabla u| \in L^{h(x)}(\Omega)\},
\end{equation*}
with the norm
\begin{equation*}
\|u\|_{1,h(x)}=|u|_{h(x)}+|\nabla u|_{h(x)},\\\ \forall u\in W^{1,h(x)}(\Omega).
\end{equation*}
where $\nabla u=(\partial_{x_{1}}u,...,\partial_{x_{N}}u) $ and $\partial_{x_{i}}=\frac{\partial }{\partial x_{i}}$ is the partial differential operator.
\begin{proposition}\label{Prop:2.4} If $1<h^{-}\leq h^{+}<\infty $, then the spaces
$L^{h(x)}(\Omega)$ and $W^{1,h(x)}(\Omega)$ are separable and reflexive Banach spaces.
\end{proposition}
The space $W_{0}^{1,h(x)}(\Omega)$ is defined as
$\overline{C_{0}^{\infty }(\Omega )}^{\|\cdot\|_{1,h(x)}}=W_{0}^{1,h(x)}(\Omega)$, and hence, it is the smallest closed set that contains $C_{0}^{\infty }(\Omega )$. Therefore, $W_{0}^{1,h(x)}(\Omega)$ is also a separable and reflexive Banach space due to the inclusion $W_{0}^{1,h(x)}(\Omega) \subset W^{1,h(x)}(\Omega)$. \\
Note that as  a consequence of Poincar\'{e} inequality, $\|u\|_{1,h(x)}$ and $|\nabla u|_{h(x)}$ are equivalent norms on $W_{0}^{1,h(x)}(\Omega)$. Therefore, for any $u\in W_{0}^{1,h(x)}(\Omega)$ we can define an equivalent norm $\|u\|$ such that
\begin{equation*}
\|u\| =|\nabla u|_{h(x)}.
\end{equation*}
\begin{proposition}\label{Prop:2.5} Let $m\in C(\overline{\Omega })$. If $1\leq m(x) <h^{\ast }(x)$ for all $x\in
\overline{\Omega }$, then the embeddings $W^{1,h(x)}(\Omega) \hookrightarrow L^{h(x)}(\Omega)$ and $W_0^{1,h(x)}(\Omega) \hookrightarrow L^{h(x)}(\Omega)$  are compact and continuous, where
$h^{\ast }( x) =\left\{\begin{array}{cc}
\frac{Nh(x)}{N-h(x)} & \text{if }h(x)<N, \\
+\infty & \text{if }h(x) \geq N.
\end{array}
\right. $
\end{proposition}

Throughout the paper, we assume the following.
\begin{itemize}
\item[$(H_1)$] $p,q\in C_+(\overline{\Omega})$ with $p^-\leq p(x)\leq p^+<q^-\leq q(x)\leq q^+<N$.
\item[$(H_2)$] $\mu\in L^\infty(\Omega)$ such that $\mu(\cdot)\geq 0$.
\end{itemize}
To address problem (\ref{e1.1a}), it is necessary to utilize the theory of the Musielak-Orlicz Sobolev space $W_0^{1,\mathcal{H}}(\Omega)$. Therefore, we subsequently introduce the double-phase operator, the Musielak-Orlicz space, and the Musielak-Orlicz Sobolev space in turn.\\

Let $\mathcal{H}:\Omega\times [0,\infty]\to [0,\infty]$ be the nonlinear function, i.e. the \textit{double phase operator}, defined by
\[
\mathcal{H}(x,t)=t^{p(x)}+\mu(x)t^{q(x)}\ \text{for all}\ (x,t)\in \Omega\times [0,\infty].
\]
Then the corresponding modular $\rho_\mathcal{H}(\cdot)$ is given by
\[
\displaystyle\rho_\mathcal{H}(u)=\int_\Omega\mathcal{H}(x,|u|)dx=
\int_\Omega\left(|u|^{p(x)}+\mu(x)|u|^{q(x)}\right)dx.
\]
The \textit{Musielak-Orlicz space} $L^{\mathcal{H}}(\Omega)$, is defined by
\[
L^{\mathcal{H}}(\Omega)=\left\{u:\Omega\to \mathbb{R}\,\, \text{measurable};\,\, \rho_{\mathcal{H}}(u)<+\infty\right\},
\]
endowed with the Luxemburg norm
\[
\|u\|_{\mathcal{H}}=\inf\left\{\zeta>0: \rho_{\mathcal{H}}\left(\frac{u}{\zeta}\right)\leq 1\right\}.
\]
Analogous to Proposition \ref{Prop:2.2}, there are similar relationship between the modular $\rho_{\mathcal{H}}(\cdot)$ and the norm $\|\cdot\|_{\mathcal{H}}$, see \cite[Proposition 2.13]{crespo2022new} for a detailed proof.

\begin{proposition}\label{Prop:2.2a}
Assume $(H_1)$ hold, and $u\in L^{\mathcal{H}}(\Omega)$. Then
\begin{itemize}
\item[$(i)$] If $u\neq 0$, then $\|u\|_{\mathcal{H}}=\zeta\Leftrightarrow \rho_{\mathcal{H}}(\frac{u}{\zeta})=1$,
\item[$(ii)$] $\|u\|_{\mathcal{H}}<1\ (\text{resp.}\ >1, =1)\Leftrightarrow \rho_{\mathcal{H}}(\frac{u}{\zeta})<1\ (\text{resp.}\ >1, =1)$,
\item[$(iii)$] If $\|u\|_{\mathcal{H}}<1\Rightarrow \|u\|_{\mathcal{H}}^{q^+}\leq \rho_{\mathcal{H}}(u)\leq \|u\|_{\mathcal{H}}^{p^-}$,
\item[$(iv)$]If $\|u\|_{\mathcal{H}}>1\Rightarrow \|u\|_{\mathcal{H}}^{p^-}\leq \rho_{\mathcal{H}}(u)\leq \|u\|_{\mathcal{H}}^{q^+}$,
\item[$(v)$] $\|u\|_{\mathcal{H}}\to 0\Leftrightarrow \rho_{\mathcal{H}}(u)\to 0$,
\item[$(vi)$]$\|u\|_{\mathcal{H}}\to +\infty\Leftrightarrow \rho_{\mathcal{H}}(u)\to +\infty$,
\item[$(vii)$] $\|u\|_{\mathcal{H}}\to 1\Leftrightarrow \rho_{\mathcal{H}}(u)\to 1$,
\item[$(viii)$] If $u_n\to u$ in $L^{\mathcal{H}}(\Omega)$, then $\rho_{\mathcal{H}}(u_n)\to\rho_{\mathcal{H}}(u)$.
\end{itemize}
\end{proposition}
The \textit{Musielak-Orlicz Sobolev space} $W^{1,\mathcal{H}}(\Omega)$ is defined by
\[
W^{1,\mathcal{H}}(\Omega)=\left\{u\in L^{\mathcal{H}}(\Omega):
|\nabla u|\in L^{\mathcal{H}}(\Omega)\right\},
\]
and equipped with the norm
\[
\|u\|_{1,\mathcal{H}}=\|\nabla u\|_{\mathcal{H}}+\|u\|_{\mathcal{H}},
\]
where $\|\nabla u\|_{\mathcal{H}}=\|\,|\nabla u|\,\|_{\mathcal{H}}$.
The space $W_0^{1,\mathcal{H}}(\Omega)$ is defined as $\overline{C_{0}^{\infty }(\Omega )}^{\|\cdot\|_{1,\mathcal{H}}}=W_0^{1,\mathcal{H}}(\Omega)$. Notice that $L^{\mathcal{H}}(\Omega), W^{1,\mathcal{H}}(\Omega)$ and $W_0^{1,\mathcal{H}}(\Omega)$ are uniformly convex and reflexive Banach spaces. Moreover, we have the following embeddings \cite[Proposition 2.16]{crespo2022new}.
\begin{proposition}\label{Prop:2.7a}
Let $(H)$ be satisfied. Then the following embeddings hold:
\begin{itemize}
\item[$(i)$] $L^{\mathcal{H}}(\Omega)\hookrightarrow L^{h(\cdot)}(\Omega), W^{1,\mathcal{H}}(\Omega)\hookrightarrow W^{1,h(\cdot)}(\Omega)$, $W_0^{1,\mathcal{H}}(\Omega)\hookrightarrow W_0^{1,h(\cdot)}(\Omega)$ are continuous for all $h\in C(\overline{\Omega})$ with $1\leq h(x)\leq p(x)$ for all $x\in \overline{\Omega}$.
\item[$(ii)$] $W^{1,\mathcal{H}}(\Omega)\hookrightarrow L^{h(\cdot)}(\Omega)$ and $W_0^{1,\mathcal{H}}(\Omega)\hookrightarrow L^{h(\cdot)}(\Omega)$ are compact for all $h\in C(\overline{\Omega})$ with $1\leq h(x)< p^*(x)$ for all $x\in \overline{\Omega}$.
\end{itemize}
\end{proposition}
As a conclusion of Proposition \ref{Prop:2.7a}, we have the continuous embedding $W_0^{1,\mathcal{H}}(\Omega)\hookrightarrow L^{h(\cdot)}(\Omega)$, and $c_{\mathcal{H}}$ denotes the best constant such that
\[
|u|_{h(\cdot)}\leq c_{\mathcal{H}}\|u\|_{1,\mathcal{H},0}.
\]
Moreover, by \cite[Proposition 2.18]{crespo2022new}, $W_0^{1,\mathcal{H}}(\Omega)$ is compactly embedded in $L^{\mathcal{H}}(\Omega)$. Thus,
$W_0^{1,\mathcal{H}}(\Omega)$ can be equipped with the equivalent norm
\[
\|u\|_{1,\mathcal{H},0}=\|\nabla u\|_{\mathcal{H}}.
\]
\begin{proposition}\label{Prop:2.7}
For the convex functional $\varrho_{\mathcal{H}}(u):=\int_\Omega\left(\frac{|\nabla u|^{p(x)}}{p(x)}+\mu(x)\frac{|\nabla u|^{q(x)}}{q(x)}\right)dx$, we have $\varrho_{\mathcal{H}}\in C^{1}(W_0^{1,\mathcal{H}}(\Omega),\mathbb{R})$ with the derivative
$$
\langle\varrho_{\mathcal{H}}^{\prime}(u),\varphi\rangle=\int_{\Omega}(|\nabla u|^{p(x)-2}\nabla u+\mu(x)|\nabla u|^{q(x)-2}\nabla u)\cdot\nabla \varphi dx,
$$
for all  $u, \varphi \in W_0^{1,\mathcal{H}}(\Omega)$, where $\langle \cdot, \cdot\rangle$ is the dual pairing between $W_0^{1,\mathcal{H}}(\Omega)$ and its dual $W_0^{1,\mathcal{H}}(\Omega)^{*}$ \cite{crespo2022new}.
\end{proposition}
\begin{remark}\label{Rem:2.1a}
Notice that
\begin{itemize}
\item[$(i)$] If $\|\nabla u\|_{\mathcal{H}}<1\Rightarrow \frac{1}{q^+}\|\nabla u\|_{\mathcal{H}}^{q^+}\leq \varrho_{\mathcal{H}}(u)\leq \frac{1}{p^-}\|\nabla u\|_{\mathcal{H}}^{p^-}$,
\item[$(ii)$]If $\|\nabla u\|_{\mathcal{H}}>1\Rightarrow \frac{1}{q^+}\|\nabla u\|_{\mathcal{H}}^{p^-}\leq \varrho_{\mathcal{H}}(u)\leq \frac{1}{p^-}\|\nabla u\|_{\mathcal{H}}^{q^+}$.
\end{itemize}
\end{remark}

We lastly introduce the seminormed space
\begin{equation}
L^{q(\cdot)}_{\mu}(\Omega)=\left\{u:\Omega\to \mathbb{R}\,\, \text{measurable};\,\, \int_\Omega\mu(x)|u|^{q(x)}dx<+\infty\right\},
\end{equation}
which is endow with the seminorm
\begin{equation}
|u|_{q(\cdot),\mu}=\inf\left\{\varsigma>0: \int_\Omega\mu(x)\left(\frac{|u|}{\varsigma}\right)^{q(x)}dx \leq 1 \right\}.
\end{equation}
We have $L^{\mathcal{H}}(\Omega)\hookrightarrow L^{q(\cdot)}_{\mu}(\Omega)$ continuously \cite[Proposition 2.16]{crespo2022new}.

\begin{proposition} \label{Prop:2.1} \cite[Proposition 2.7]{avci2019positive}
Let $X$ be a vector space, and let $I:X \rightarrow \mathbb{R}$. Then $I$ is convex if and only if
\begin{equation}\label{e3.3bc}
I((1-\epsilon)u+\epsilon v)\leq (1-\epsilon)\omega_1+\epsilon \omega_2, \,\,\,\, 0<\epsilon<1,
\end{equation}
whenever $I(u)<\omega_1$ and $I(v)<\omega_2$, for all $u,v \in X$ and $\omega_1,\omega_2 \in \mathbb{R}$.
\end{proposition}
\begin{proof}
{The proof was originally established by the author in his earlier work (see \cite{avci2019positive,avci2024existence}). Assume that functional $I:X \rightarrow \mathbb{R}$ is convex. Since $I$ is a real-valued functional, there are real numbers $\theta,\delta \in \mathbb{R}$ such that $I(u)<\theta$ and $I(v)<\delta$. Then,}
\begin{equation*}
I((1-\lambda)u+\lambda v)<(1-\lambda)I(u)+\lambda I(v) <(1-\lambda)\theta+\lambda \delta, \,\,\,\, 0<\lambda<1.
\end{equation*}
{On} the other hand, assume that (\ref{e3.3bc}) holds. Since $I(u)<\theta$ and $I(v)<\delta$, we can write %Author: confirmed.
$$
I(u)<I(u)+\varepsilon:=\theta,
$$
$$
I(v)<I(v)+\varepsilon:=\delta,
$$
for a real number $\varepsilon >0$. Therefore,
\begin{equation}\label{e3.3bdd}
I((1-\lambda)u+\lambda v)<(1-\lambda)I(u)+\lambda I(v)+\varepsilon, \,\,\,\, 0<\lambda<1.
\end{equation}
{However}, considering that $\varepsilon >0$ is arbitrary, we conclude that %Author: confirmed.
\begin{equation*}
I((1-\lambda)u+\lambda v)\leq(1-\lambda)I(u)+\lambda I(v).
\end{equation*}
\end{proof}

\begin{proposition} \label{Prop:2.2bc}
Let $x,y \in \mathbb{R}^N$ and let $|\cdot|$ be the Euclidean norm in $\mathbb{R}^N$. Then for any $1\leq p<\infty$ and the real parameters $a,b>0$ it holds
\begin{equation}\label{e2.1a1}
\frac{|a |x|^{p-2}x-b |y|^{p-2}y| }{||x|^{p-2}x-|y|^{p-2}y|}\leq  \left(a+|a-b|\right)+\frac{|a-b|}{||x|^{p-2}x-|y|^{p-2}y|}.
\end{equation}
\end{proposition}
\begin{proof}
If $a=b$, then there is nothing to do. So, we assume that $a \neq b$.\\
Put
\begin{equation}\label{e2.1a2}
\Lambda(x,y)=\frac{|a |x|^{p-2}x-b |y|^{p-2}y|}{||x|^{p-2}x-|y|^{p-2}y|}.
\end{equation}
Notice that $\Lambda$ is invariant by any orthogonal transformation $T$; that is, $\Lambda(Tx,Ty)=\Lambda(x,y)$ for all $x,y \in \mathbb{R}^N$.

Thus, using this fact, along with  the homogeneity of $\Lambda$,  we can let $x=|x|e_1$ and assume that $x=e_1$. Then,  without loss of generality, it is enough to work with the function
\begin{equation}\label{e2.1a5}
\Lambda(e_1,y)=\frac{|a e_1 -b |y|^{p-2}y|}{|e_1-|y|^{p-2}y|}.
\end{equation}
In doing so, first, we have
\begin{align}\label{e2.1a6}
|a e_1 -b |y|^{p-2}y|& =a\bigg|e_1 -\frac{b}{a} |y|^{p-2}y\bigg|=a\bigg|(e_1-|y|^{p-2}y) +\left(1-\frac{b}{a}\right) |y|^{p-2}y\bigg| \nonumber\\
& \leq a|e_1-|y|^{p-2}y|+|a-b|||y|^{p-2}y| \nonumber\\
& \leq a|e_1-|y|^{p-2}y|+|a-b|\left(|e_1-|y|^{p-2}y|+|e_1|\right) \nonumber\\
& \leq a|e_1-|y|^{p-2}y|+|a-b||e_1-|y|^{p-2}y|+|a-b| \nonumber\\
& \leq |e_1-|y|^{p-2}y|\left(a+|a-b|\right)+|a-b|.
\end{align}
Then using this in (\ref{e2.1a5}) we obtain
\begin{align}\label{e2.1a7}
\Lambda(e_1,y)&\leq\frac{|e_1-|y|^{p-2}y|\left(a+|a-b|\right)+|a-b|}{|e_1-|y|^{p-2}y|} \nonumber\\
&\leq a+|a-b|+\frac{|a-b|}{|e_1-|y|^{p-2}y|}
\end{align}
which shows that (\ref{e2.1a1}) hold.
\end{proof}

\section{The main results}\label{sec3}

The energy functional $\mathcal{I}:W_0^{1,\mathcal{H}}(\Omega)\rightarrow \mathbb{R}$ corresponding to equation (\ref{e1.1a}) by
\begin{align}\label{e3.1}
\mathcal{I}_{\lambda}(u)&= \widehat{\mathcal{M}}\left(\int_{\Omega}\left(\frac{|\nabla u|^{p(x)}}{p(x)}+\mu(x)\frac{|\nabla u|^{q(x)}}{q(x)}\right)dx\right)-\lambda\int_{\Omega}F(x,u)dx,
\end{align}
where $F(x,t)=\int_{0}^{t}f(x,s)ds$, and $\widehat{\mathcal{M}}(t)=\int^{t}_{0}\mathcal{M}(s)ds$.\\
\begin{definition}\label{Def:3.1} A function $u\in W_0^{1,\mathcal{H}}(\Omega)$ is called a weak solution to problem (\ref{e1.1a}) if
\begin{align}\label{e3.2}
&{\mathcal{M}}_{\varrho}(u)\int_{\Omega}(|\nabla u|^{p(x)-2}\nabla u+\mu(x)|\nabla u|^{q(x)-2}\nabla u)\cdot \nabla \varphi dx\nonumber\\
&=\lambda\int_{\Omega}f(x,u)\varphi dx,\,\,\ \forall \varphi\in W_0^{1,\mathcal{H}}(\Omega),
\end{align}
\end{definition}
where $\mathcal{M}(\varrho_{\mathcal{H}}(u)):=\mathcal{M}_{\varrho}(u)$. It is well-known that the critical points of the functional $\mathcal{I}_{\lambda}$ corresponds to the weak solutions of problem (\ref{e1.1a}).\\
Let's define the functionals $\mathcal{J},\mathcal{K}:W_0^{1,\mathcal{H}}(\Omega) \to \mathbb{R}$ by
\begin{equation}\label{e3.3}
\mathcal{J}(u):=\int_{\Omega}F(x,u)dx,
\end{equation}
and
\begin{equation}\label{e3.3a}
\mathcal{K}(u):=\widehat{\mathcal{M}}(\varrho_{\mathcal{H}}(u))=\widehat{\mathcal{M}}_{\varrho}(u).
\end{equation}
Then,
\begin{equation}\label{e3.3b}
\mathcal{I}_{\lambda}(\cdot):=\mathcal{K}(\cdot)-\lambda \mathcal{J}(\cdot).
\end{equation}

To obtain the main result, we apply the following well-known critical point result given by Bonanno and and Marano.

\begin{lemma}\label{Lem:3.1}\cite[Theorem 3.6]{bonanno2010structure}
Let $X$ be a reflexive real Banach space, $\Phi:X\rightarrow \mathbb{R}$ be a coercive, continuously G\^{a}teaux differentiable and sequentially weakly lower semi-continuous functional whose G\^{a}teaux
derivative admits a continuous inverse on $X^*$; $\Psi:X\rightarrow \mathbb{R}$ be a continuously G\^{a}teaux differentiable functional whose G\^{a}teaux derivative is compact
such that
\begin{equation}\label{e3.4}
\inf_{x\in X}\Phi (x)=\Phi (0)=\Psi (0)=0.
\end{equation}
Assume that there exist $r>0$ and $\overline{x}\in X$, with $r<\Phi(\overline{x})$, such that:
\begin{itemize}
\item[(a1)] $\frac{1}{r}\sup_{\Phi (x)\leq r}\Psi (x)\leq \frac{\Psi (\overline{x})}{\Phi (\overline{x})}$;
\item[(a2)] \textit{for each }$\lambda \in \Lambda _{r}:=\left( \frac{\Phi (\overline{x})}{\Psi (\overline{x})},\frac{r}{\sup_{\Phi (x)\leq r}\Psi (x)}\right)$, the functional $\Phi-\lambda \Psi$ is coercive.
\end{itemize}
Then for each $\lambda \in \Lambda_{r}$, the functional $\Phi-\lambda \Psi$ has at least three distinct critical points in $X$.
\end{lemma}

We assume the following.
\begin{itemize}
\item[$(f_{1})$] $f:\Omega \times \mathbb{R} \to \mathbb{R}$ is a Carathéodory function and there exist real parameters $\bar{c}_{1}, \bar{c}_{2}\geq0$ satisfying
\begin{equation}\label{e3.5}
|f(x,t)| \leq  \bar{c}_{1}+ \bar{c}_{2}|t|^{s(x)-1},\\\ \forall(x,t) \in \Omega \times \mathbb{R}
\end{equation}
where $s\in C(\overline{\Omega })$ with $1< s(x) <p^{\ast }(x)$ for all $x\in\overline{\Omega}$;
\item[$(f_{2})$] there exists a real parameter $\bar{c}_{3}\geq0$ and $r \in C_+(\overline{\Omega})$ with $r^-\leq r(x)\leq r^+<p^-$ satisfying
\begin{equation}\label{e3.6}
F(x,t)\leq  \bar{c}_{3}(1+|t|^{r(x)}),\\\ \forall(x,t) \in \Omega \times \mathbb{R};
\end{equation}
\item[$(f_{3})$] $F(x,t)\geq 0$ for all $(x,t) \in \Omega \times \mathbb{R_+}$;
\item[$(M)$] $\mathcal{M}:(0,\infty)\to [m_0, m^0)$ is a $C^1$-continuous nondecreasing function satisfying
    \begin{equation}\label{e3.7}
    \kappa_1t^{\alpha_1-1}\leq \mathcal{M}(t)\leq \kappa_2t^{\alpha_2-1},
    \end{equation}
    where $m_0, m^0,\kappa_1,\kappa_2,\alpha_1,\alpha_2$ are positive real parameters such that $\kappa_2\geq \kappa_1$ and $\alpha_2\geq \alpha_1>1$.
\end{itemize}

The main result of the paper is:

\begin{theorem}\label{Thm:4.1}
Assume $(f_{1})-(f_{3})$ and $(M)$ are satisfied. Assume also that
\begin{itemize}
 \item[$(f_{4})$] there exists positive real parameters $r,\delta$ with \newline $r<\frac{\kappa_{1}}{\alpha_{1} (q^+)^{\alpha_{1}}} \omega_N^{\alpha_{1}} R^{N\alpha_{1}}(2^{N}-1)^{\alpha_{1}}2^{-N\alpha_{1}}\left(\frac{2\delta}{R}\right)^{\alpha_{1}(p^- \wedge p^+)}$ satisfying
\begin{align}\label{e3.7a}
\sigma_r:&=\biggl\{\bar{c}_{1}c_{\mathcal{H}}(q^+)^{\frac{1}{p^-}}\left(\frac{\alpha_{1}}{\kappa_{1}}\right)^{\frac{1}{\alpha_1(p^- \wedge q^+)}}  r^{\frac{1}{\alpha_1(p^- \wedge q^+)}} + \bar{c}_{2}c_{\mathcal{H}}^{s^+}(q^+)^{\frac{s^+}{p^-}} \left(\frac{\alpha_{1}}{\kappa_{1}}\right)^{\frac{s^- \wedge s^+}{\alpha_1(p^- \wedge q^+)}}  r^{\frac{s^- \wedge s^+}{\alpha_1(p^- \wedge q^+)}}\biggl\}\nonumber \\
&<\frac{2^{N(\alpha_{2}-1)}\alpha_{2} (p^-)^{\alpha_{2}}\inf_{x \in \Omega}F(x,\delta)}{\kappa_{2}(1+|\mu|_{\infty})^{\alpha_{2}}\omega_N^{\alpha_{2}-1}(2^{N}-1)^{\alpha_{2}}R^{N(\alpha_{2}-1)}\left(\frac{2\delta}{R}\right)^{\alpha_{2}(p^- \vee q^+)}}:=\sigma^r.
\end{align}
\end{itemize}
Then for any $\lambda \in \lambda_{r,\delta}:=\left(\frac{1}{\sigma^r},\frac{1}{\sigma_r} \right)$, the problem (\ref{e1.1a}) admits at least three distinct weak solutions.
\end{theorem}

First, we obtain the regularity results of the functionals $\mathcal{I}_{\lambda}, \mathcal{K}$ and $\mathcal{J}$, which are needed to apply Lemma \ref{Lem:3.1}.\\

\begin{lemma}\label{Lem:4.2}
Assume that $(M)$ holds. Then the following holds:
\begin{itemize}
  \item [$(i)$] $\mathcal{K}$ is coercive.
  \item [$(ii)$] $\mathcal{K}$ is G\^{a}teaux differentiable whose G\^{a}teaux derivative $\mathcal{K}^{\prime}:W_0^{1,\mathcal{H}}(\Omega) \to W_0^{1,\mathcal{H}}(\Omega)^*$ given by the formula for $u, \varphi \in W_0^{1,p(x)}(\Omega)$
  \begin{equation}\label{e4.12ab}
  \langle \mathcal{K}^{\prime}(u),\varphi \rangle=\mathcal{M}_{\varrho}(u)\int_{\Omega}(|\nabla u|^{p(x)-2}\nabla u+\mu(x)|\nabla u|^{q(x)-2}\nabla u)\cdot \nabla \varphi dx,
  \end{equation}
   is bounded and continuous.
  \item [$(iii)$] $\mathcal{K}^{\prime}$ is strictly monotone.
  \item [$(iv)$] $\mathcal{K}^{\prime}$ satisfies the $(S_+)$-property; that is,
  \begin{equation*}
  \text{if } u_{n} \rightharpoonup u \in W_0^{1,\mathcal{H}}(\Omega) \text{ and } \limsup_{n\rightarrow\infty}\langle \mathcal{K}^{\prime}(u_{n}),u_{n}-u\rangle\leq 0, \text{ then } u_{n} \to u \in W_0^{1,\mathcal{H}}(\Omega).
  \end{equation*}
  \item [$(v)$] $\mathcal{K}^{\prime}$ is coercive and a homeomorphism.
  \item [$(vi)$] $\mathcal{K}$  is sequentially weakly lower semi-continuous.
\end{itemize}
\end{lemma}

\begin{proof}
$(i)$ Using $(M)$ and Proposition \ref{Prop:2.2a}, it reads
\begin{equation}\label{e3.8}
\mathcal{K}(u)=\int^{\varrho_{\mathcal{H}}(u)}_{0}\mathcal{M}(s)ds \geq \frac{\kappa_{1}}{\alpha_{1} (q^+)^{\alpha_{1}}}\|u\|_{1,\mathcal{H},0}^{\alpha_{1} p^{-}},
\end{equation}
which implies that $\mathcal{K}$  is coercive.\\
$(ii)$ Using the mean value theorem, it reads
\begin{equation}\label{e3.11d}
\begin{split}
\langle \mathcal{K}^{\prime}(u),\varphi \rangle&\\
=&\lim_{t \to 0}\biggl(\mathcal{M}_{\varrho}(u+\delta t\varphi) \times \int_{\Omega}\bigg(|\nabla(u+\delta t\varphi)|^{p(x)-2}\nabla(u+\delta t\varphi)\\
&+\mu(x)|\nabla(u+\delta t\varphi)|^{q(x)-2}\nabla(u+\delta t\varphi)\biggl) \cdot \nabla \varphi dx\biggl), \\
\end{split}
\end{equation}
for all $u, \varphi \in W_0^{1,\mathcal{H}}(\Omega)$, and $0\leq \delta \leq 1$.\\
Let
\begin{equation}\label{e3.11e}
\begin{split}
  \Theta(u+\delta t\varphi) :=& \biggl(|\nabla(u+\delta t\varphi)|^{p(x)-2}\nabla(u+\delta t\varphi) \\
  &+\mu(x)|\nabla(u+\delta t\varphi))|^{q(x)-2}\nabla(u+\delta t\varphi)\biggl) \cdot \nabla \varphi.
\end{split}
\end{equation}
Applying the Young's inequality, we obtain
\begin{align}\label{e3.11f}
  |\Theta(u+\delta t\varphi)| & \leq \hat{c}\biggl(|\nabla u|^{p(x)}+|\nabla \varphi|^{p(x)} + |\mu|^{q_{+}}_{\infty}|\nabla u|^{q(x)}+ |\nabla \varphi|^{q(x)}\biggl),  \nonumber \\
\end{align}
where $\hat{c}:=\frac{2^{q^+}(q^+-1)+1}{2q^-}$. Thanks to the relation $L^{q(x)}(\Omega) \subset L^{p(x)}(\Omega)\subset L^{1}(\Omega)$ the right-hand side of (\ref{e3.11f}) is summable over $\Omega$. Therefore, using the assumption $(M)$ and the Lebesgue's dominated convergence theorem together provides
\begin{align}\label{e3.11g}
\langle \mathcal{K}^{\prime}(u),\varphi \rangle & =\left(\lim_{t \to 0}\mathcal{M}_{\varrho}(u+\delta t\varphi)
\times \int_{\Omega}\lim_{t \to 0}\{|\nabla (u+\delta t\varphi)|^{p(x)-2}\nabla(u+\delta t\varphi)  \right. \nonumber \\
&\left. +\mu(x)|\nabla(u+\delta t\varphi))|^{q(x)-2}\nabla(u+\delta t\varphi)\}\cdot\nabla \varphi dx\right) \nonumber \\
& =\mathcal{M}_{\varrho}(u)\int_{\Omega}(|\nabla u|^{p(x)-2}\nabla u+\mu(x)|\nabla u|^{q(x)-2}\nabla u)\cdot \nabla \varphi dx.
\end{align}
The linearity of the formula (\ref{e3.11g}) follows due to the linearity of integration, and the fact it is linear in $\nabla \varphi$.\\
Next, we show that $\mathcal{K}^{\prime}$ with the formula (\ref{e3.11g}) is bounded. Then using H\"{o}lder's inequality, Proposition \ref{Prop:2.2bb}, Remark \ref{Rem:2.1a}, and the involved embeddings it follows
\begin{align}\label{e3.11ga}
|\langle \mathcal{K}^{\prime}(u),\varphi \rangle| & =\bigg|\mathcal{M}_{\varrho}(u)\int_{\Omega}(|\nabla u|^{p(x)-2}\nabla u+\mu(x)|\nabla u|^{q(x)-2}\nabla u)\cdot \nabla \varphi dx\bigg| \nonumber \\
& \leq \frac{\kappa_{2}}{(p^-)^{\alpha_{2}-1}}\|u\|_{1,\mathcal{H},0}^{(\alpha_{2}-1) p^{-}\vee q^{+}}\left(||\nabla u|^{p(x)-1}|_{\frac{p(x)}{p(x)-1}}|\nabla \varphi|_{p(x)} \right. \nonumber \\
& \left. +|\mu|_{\infty}||\nabla u|^{q(x)-1}|_{\frac{q(x)}{q(x)-1}}|\nabla \varphi|_{q(x)} \right)\nonumber \\
& \leq \frac{\kappa_{2}}{(p^-)^{\alpha_{2}-1}}\|u\|_{1,\mathcal{H},0}^{(\alpha_{2}-1) p^{-}\vee q^{+}}\left(\|u\|_{1,\mathcal{H},0}^{p_{+}-1}+|\mu|_{\infty}\|u\|_{1,\mathcal{H},0}^{q_{+}-1} \right)\|\varphi\|_{1,\mathcal{H},0}\nonumber \\
& \leq \frac{2|\mu|_{\infty}\kappa_{2}}{(p^-)^{\alpha_{2}-1}}\|u\|_{1,\mathcal{H},0}^{\tau}\|\varphi\|_{1,\mathcal{H},0}.
\end{align}
 where $\tau:=\max\{(\alpha_{2}-1)(p^{-}\vee q^{+})+(p_{+}-1), (\alpha_{2}-1)(p^{-}\vee q^{+})+(q_{+}-1)\}$. Therefore,
\begin{align}\label{e3.11gb}
\|\mathcal{K}^{\prime}(u)\|_{W_0^{1,\mathcal{H}}(\Omega)^*}=\sup_{\varphi \in W_0^{1,\mathcal{H}}(\Omega),\|\varphi\|_{1,\mathcal{H},0}\leq 1}|\langle \mathcal{K}^{\prime}(u),\varphi \rangle| \leq \frac{|\mu|_{\infty}\kappa_{2}}{(p^-)^{\alpha_{2}-1}}\|u\|_{1,\mathcal{H},0}^{\tau}, \quad \forall u \in W_0^{1,\mathcal{H}}(\Omega)
\end{align}
hence $\mathcal{K}^{\prime}$ is bounded. Therefore, $\mathcal{K}$ is G\^{a}teaux differentiable with the derivative given by the formula (\ref{e3.11g}).\\

Next, we continue with the continuity of $\mathcal{K}^{\prime}$. In doing so, for a sequence $(u_n) \subset W_0^{1,\mathcal{H}}(\Omega)$, assume that $u_n \to u$ in $W_0^{1,\mathcal{H}}(\Omega)$. Then using Proposition \ref{Prop:2.2bc}, we have
\begin{align}\label{e3.11h}
&|\langle \mathcal{K}^{\prime}(u_n)-\mathcal{K}^{\prime}(u),v \rangle| \nonumber \\
& \leq \int_{\Omega}\bigg|\mathcal{M}_{\varrho}(u_n)|\nabla u_n|^{p(x)-2}\nabla u_n-\mathcal{M}_{\varrho}(u)|\nabla u|^{p(x)-2}\nabla u\bigg| |\nabla v| dx \nonumber \\
& + \int_{\Omega}\mu(x)\bigg|\mathcal{M}_{\varrho}(u_n)|\nabla u_n|^{q(x)-2}\nabla u_n-\mathcal{M}_{\varrho}(u)|\nabla u|^{q(x)-2}\nabla u\bigg| |\nabla v| dx \nonumber \\
& \leq \int_{\Omega}\left\{||\nabla u_n|^{p(x)-2}\nabla u_n-|\nabla u|^{p(x)-2}\nabla u|\times (G_{n}+\mathcal{M}_{\varrho}(u_n))+G_{n}\right\} |\nabla v| dx \nonumber \\
& + \int_{\Omega}\mu(x)\left\{||\nabla u_n|^{q(x)-2}\nabla u_n-|\nabla u|^{q(x)-2}\nabla u|\times (G_{n}+\mathcal{M}_{\varrho}(u_n))+G_{n}\right\} |\nabla v| dx \nonumber \\
& \leq (G_{n}+\mathcal{M}_{\varrho}(u_n))\int_{\Omega}||\nabla u_n|^{p(x)-2}\nabla u_n-|\nabla u|^{p(x)-2}\nabla u| |\nabla v| dx + G_{n}\int_{\Omega} |\nabla v| dx\nonumber \\
& + (G_{n}+\mathcal{M}_{\varrho}(u_n))\int_{\Omega}\mu(x)||\nabla u_n|^{q(x)-2}\nabla u_n-|\nabla u|^{q(x)-2}\nabla u||\nabla v| dx +G_{n}\int_{\Omega} |\nabla v| dx.
\end{align}
where $G_{n}:=|\mathcal{M}_{\varrho}(u_n)-\mathcal{M}_{\varrho}(u)|$.
Now, if we apply H\"{o}lder's inequality and consider the embedding $L^{\mathcal{H}}(\Omega)\hookrightarrow L^{q(x)}_{\mu}(\Omega)$ and Propositions \ref{Prop:2.2a} and \ref{Prop:2.7a}, it reads
\begin{align}\label{e3.11k}
&|\langle \mathcal{K}^{\prime}(u_n)-\mathcal{K}^{\prime}(u),v \rangle| \nonumber \\
& \leq (G_{n}+\mathcal{M}_{\varrho}(u_n)) \bigg| ||\nabla u_n|^{p(x)-2}\nabla u_n-|\nabla u|^{p(x)-2}\nabla u| \bigg|_{\frac{p(x)}{p(x)-1}}|\nabla v|_{p(x)}+G_{n}|\Omega||\nabla v|_{p(x)}\nonumber \\
&+(G_{n}+\mathcal{M}_{\varrho}(u_n))\bigg|\mu(x)^{\frac{q(x)-1}{q(x)}} ||\nabla u_n|^{q(x)-2}\nabla u_n-|\nabla u|^{q(x)-2}\nabla u| \bigg|_{\frac{q(x)}{q(x)-1}}|\mu(x)^{\frac{1}{q(x)}}|\nabla v||_{q(x)} \nonumber \\
& +G_{n}|\Omega||\nabla v|_{p(x)},
\end{align}
and therefore
\begin{align}\label{e3.11hh}
\|\mathcal{K}^{\prime}(u_n)-\mathcal{K}^{\prime}(u)\|_{W_0^{1,\mathcal{H}}(\Omega)^*}=\sup_{v \in W_0^{1,\mathcal{H}}(\Omega),\|v\|_{1,\mathcal{H},0}\leq 1}|\langle \mathcal{K}^{\prime}(u_n)-\mathcal{K}^{\prime}(u),v \rangle| \to 0.
\end{align}

Note that the result (\ref{e3.11hh}) follows due to the following:\\
Since $u_n \to u$ in $W_0^{1,\mathcal{H}}(\Omega)$, by the embeddings $L^{\mathcal{H}}(\Omega)\hookrightarrow L^{q(x)}_{\mu}(\Omega)$, $W_0^{1,\mathcal{H}}(\Omega) \hookrightarrow L^{\mathcal{H}}(\Omega)$ and $W_0^{1,\mathcal{H}}(\Omega) \hookrightarrow L^{p(x)}(\Omega)$ we have
\begin{equation}\label{e3.11hk}
\lim_{n \to \infty}\int_{\Omega}|\nabla u_n|^{p(x)}dx=\int_{\Omega}|\nabla u|^{p(x)}dx,
\end{equation}
and
\begin{equation}\label{e3.11hm}
\lim_{n \to \infty}\int_{\Omega}\mu(x)|\nabla u_n|^{q(x)}dx=\int_{\Omega}\mu(x)|\nabla u|^{q(x)}dx.
\end{equation}
However, by Vitali’s Theorem \cite[Theorem 4.5.4]{bogachev2007measure}, (\ref{e3.11hk}) and  (\ref{e3.11hm}) means $|\nabla u_n| \to |\nabla u|$ and $\mu(x)^{\frac{1}{q(x)}}|\nabla u_n| \to \mu(x)^{\frac{1}{q(x)}}|\nabla u|$ in measure in $\Omega$  and the sequences $\left\{|\nabla u_n|^{p(x)}\right\}_{n}$ and $\left\{\mu(x)|\nabla u_n|^{q(x)}\right\}_{n}$ are uniformly integrable over $\Omega$. Now, lets' consider the inequalities

\begin{equation}\label{e3.11hn}
||\nabla u_n|^{p(x)-2}\nabla u_n-|\nabla u|^{p(x)-2}\nabla u|^{\frac{p(x)}{p(x)-1}}\leq 2^{\frac{p^+}{p^--1}-1}(|\nabla u_n|^{p(x)}+|\nabla u|^{p(x)}),
\end{equation}
and
\begin{equation}\label{e3.11hnm}
\mu(x)||\nabla u_n|^{q(x)-2}\nabla u_n-|\nabla u|^{q(x)-2}\nabla u|^{\frac{q(x)}{q(x)-1}}\leq 2^{\frac{q^+}{q^--1}-1}\mu(x)(|\nabla u_n|^{q(x)}+|\nabla u|^{q(x)}).
\end{equation}
Therefore, the families $\left\{||\nabla u_n|^{p(x)-2}\nabla u_n-|\nabla u|^{p(x)-2}\nabla u|^{\frac{p(x)}{p(x)-1}}\right\}_{n}$ and $\left\{\mu(x)||\nabla u_n|^{q(x)-2}\nabla u_n-|\nabla u|^{q(x)-2}\nabla u|^{\frac{q(x)}{q(x)-1}}\right\}_{n}$ are also uniformly integrable over $\Omega$. Thus, by Vitali’s Theorem, $|\nabla u|^{p(x)-2}\nabla u$ and $\mu(x)|\nabla u|^{q(x)-2}\nabla u$ are integrable and $|\nabla u_n|^{p(x)-2}\nabla u_n \to |\nabla u|^{p(x)-2}\nabla u$ and $\mu(x)|\nabla u_n|^{q(x)-2}\nabla u_n \to \mu(x)|\nabla u|^{q(x)-2}\nabla u$ in $L^1(\Omega)$, and hence
\begin{equation}\label{e3.11hnk}
 \bigg| ||\nabla u_n|^{p(x)-2}\nabla u_n-|\nabla u|^{p(x)-2}\nabla u| \bigg|_{\frac{p(x)}{p(x)-1}} \to 0,
\end{equation}
and
\begin{equation}\label{e3.11hns}
\bigg|\mu(x)^{\frac{q(x)-1}{q(x)}} ||\nabla u_n|^{q(x)-2}\nabla u_n-|\nabla u|^{q(x)-2}\nabla u| \bigg|_{\frac{q(x)}{q(x)-1}} \to 0.
\end{equation}
Lastly, by the assumption $(M)$ and Proposition \ref{Prop:2.7}, $(G_{n}+\mathcal{M}_{\varrho}(u_n))$ is bounded for $n=1,2,...$, and
\begin{equation}\label{e3.11ha}
G_{n}=|\mathcal{M}_{\varrho}(u_n)-\mathcal{M}_{\varrho}(u)|\to 0 \text{ as } n \to \infty.
\end{equation}
Therefore the result (\ref{e3.11hh}) follows. Putting all these together, we infer that $\mathcal{K}^{\prime}$ is continuously G\^{a}teaux differentiable whose derivative is given by the formula (\ref{e4.12ab}).\\

$(iii)$ Now, we show that $\mathcal{K}^{\prime}$ is strictly monotone. To do so, we argue similarly to \cite{massar2024existence}.
Let $u,v \in W_0^{1,\mathcal{H}}(\Omega)$ with $u \neq v$. Without loss of generality, we can assume that $\varrho_{\mathcal{H}}(u)\geq \varrho_{\mathcal{H}}(v)$. Then, $\mathcal{M}_{\varrho}(u)\geq\mathcal{M}_{\varrho}(v)$ due to $(M)$ and Proposition \ref{Prop:2.7}.\\
Noticing that
\begin{equation}\label{e3.11ka}
0 \leq (\nabla u- \nabla v)^{2} \Rightarrow\nabla u \cdot \nabla v\leq 2^{-1}(|\nabla u|^{2}+|\nabla v|^{2}),
\end{equation}
we obtain
\begin{align}\label{e3.11kb}
\langle\varrho_{\mathcal{H}}^{\prime}(u),u-v\rangle&=\int_{\Omega}(|\nabla u|^{p(x)-2}\nabla u+\mu(x)|\nabla u|^{q(x)-2}\nabla u)\cdot \nabla(u-v) dx \nonumber \\
&=\int_{\Omega}\left(|\nabla u|^{p(x)}+\mu(x)|\nabla u|^{q(x)}-(|\nabla u|^{p(x)-2}+\mu(x)|\nabla u|^{q(x)-2}) \nabla u \cdot \nabla v\right)dx \nonumber \\
&\geq  2^{-1}\int_{\Omega}(|\nabla u|^{p(x)-2}+\mu(x)|\nabla u|^{q(x)-2})(|\nabla u|^{2}-|\nabla v|^{2})dx,
\end{align}
and
\begin{align}\label{e3.11kc}
\langle\varrho_{\mathcal{H}}^{\prime}(v),v-u\rangle&=\int_{\Omega}(|\nabla v|^{p(x)-2}\nabla v+\mu(x)|\nabla v|^{q(x)-2}\nabla v)\cdot \nabla(v-u) dx \nonumber \\
&=\int_{\Omega}\left(|\nabla v|^{p(x)}+\mu(x)|\nabla v|^{q(x)}-(|\nabla v|^{p(x)-2}+\mu(x)|\nabla v|^{q(x)-2}) \nabla u \cdot \nabla v\right)dx \nonumber \\
&\geq  2^{-1}\int_{\Omega}(|\nabla v|^{p(x)-2}+\mu(x)|\nabla v|^{q(x)-2})(|\nabla v|^{2}-|\nabla u|^{2})dx.
\end{align}
Next, we partition $\Omega$ into $\Omega_{\geq}=\{x \in \Omega: |\nabla u|\geq |\nabla v| \}$ and $\Omega_{<}=\{x \in \Omega: |\nabla u|< |\nabla v| \}$. Hence, using (\ref{e3.11kb}), (\ref{e3.11kc}) and $(M)$, we have
\begin{align}\label{e3.11kd}
I_{1}(u)&:=\mathcal{M}_{\varrho}(u)\langle\varrho_{\mathcal{H}}^{\prime}(u),u-v\rangle\nonumber \\
&=\mathcal{M}_{\varrho}(u)\int_{\Omega_{\geq}}\left(|\nabla u|^{p(x)}+\mu(x)|\nabla u|^{q(x)}+\mu(x)|\nabla u|^{r(x)}-(|\nabla u|^{p(x)-2}+\mu(x)|\nabla u|^{q(x)-2}) \nabla u \cdot \nabla v\right)dx \nonumber \\
& \geq \frac {\mathcal{M}_{\varrho}(u)}{2}\int_{\Omega_{\geq}}(|\nabla u|^{p(x)-2}+\mu(x)|\nabla u|^{q(x)-2})(|\nabla u|^{2}-|\nabla v|^{2})dx\nonumber\\
&- \frac {\mathcal{M}_{\varrho}(v)}{2}\int_{\Omega_{\geq}}(|\nabla v|^{p(x)-2}+\mu(x)|\nabla v|^{q(x)-2})(|\nabla u|^{2}-|\nabla v|^{2})dx\nonumber\\
& \geq \frac {\mathcal{M}_{\varrho}(v)}{2}\int_{\Omega_{\geq}}\left((|\nabla u|^{p(x)-2}+\mu(x)|\nabla u|^{q(x)-2})- (|\nabla v|^{p(x)-2}+\mu(x)|\nabla v|^{q(x)-2})\right) \nonumber \\
& \times (|\nabla u|^{2}-|\nabla v|^{2})dx\nonumber\\
& \geq \frac {m_0}{2}\int_{\Omega_{\geq}}\left((|\nabla u|^{p(x)-2}+\mu(x)|\nabla u|^{q(x)-2})- (|\nabla v|^{p(x)-2}+\mu(x)|\nabla v|^{q(x)-2})\right)\nonumber\\
& \times (|\nabla u|^{2}-|\nabla v|^{2})dx\nonumber\\
& \geq 0,
\end{align}
and similarly
\begin{align}\label{e3.11ke}
I_{2}(v)&:=\mathcal{M}_{\varrho}(v)\langle\varrho_{\mathcal{H}}^{\prime}(v),v-u\rangle\nonumber \\
&=\mathcal{M}_{\varrho}(v)\int_{\Omega_{<}}\left(|\nabla v|^{p(x)}+\mu(x)|\nabla v|^{q(x)}-(|\nabla v|^{p(x)-2}+\mu(x)|\nabla v|^{q(x)-2}) \nabla u \cdot \nabla v\right)dx \nonumber \\
&\geq \frac {m_0}{2}\int_{\Omega_{<}}\left((|\nabla u|^{p(x)-2}+\mu(x)|\nabla v|^{q(x)-2})- (|\nabla v|^{p(x)-2}+\mu(x)|\nabla v|^{q(x)-2})\right)\nonumber\\
& \times (|\nabla u|^{2}-|\nabla v|^{2})dx \nonumber \\
& \geq 0.
\end{align}
Note that
\begin{align}\label{e3.11kf}
\langle \mathcal{K}^{\prime}(u)-\mathcal{K}^{\prime}(v),u-v \rangle =\langle \mathcal{K}^{\prime}(u),u-v \rangle+\langle \mathcal{K}^{\prime}(v),v-u \rangle = I_{1}(u)+I_{2}(v)\geq 0.
\end{align}
However, we discard the case of $\langle \mathcal{K}^{\prime}(u)-\mathcal{K}^{\prime}(v),u-v \rangle =0$ since this would eventually imply that $\partial_{x_i}u=\partial_{x_i}v$ for $i=1,2,...,N$, which would contradict the assumption that $u \neq v $ in $W_0^{1,\mathcal{H}}(\Omega)$. Thus,
\begin{align}\label{e3.11kg}
\langle \mathcal{K}^{\prime}(u)-\mathcal{K}^{\prime}(v),u-v \rangle >0.
\end{align}

$(iv)$ Let $(u_{n})\subset W_0^{1,\mathcal{H}}(\Omega)$ be a sequence satisfying
\begin{equation}\label{e3.2a}
u_{n} \rightharpoonup u \in W_0^{1,\mathcal{H}}(\Omega),
\end{equation}
\begin{equation}\label{e3.2b}
\limsup_{n\rightarrow\infty}\langle \mathcal{K}^{\prime}(u_{n}),u_{n}-u\rangle\leq 0.
\end{equation}
We shall show that $u_{n} \to u \in W_0^{1,\mathcal{H}}(\Omega)$. \\
By (\ref{e3.2a}), we have
\begin{equation}\label{e3.2c}
\lim_{n\rightarrow\infty}\langle \mathcal{K}^{\prime}(u),u_{n}-u\rangle=0.
\end{equation}
Using By (\ref{e3.2b}) and By (\ref{e3.2c}), it reads
\begin{equation}\label{e3.2d}
\limsup_{n\rightarrow\infty}\langle \mathcal{K}^{\prime}(u_{n})-\mathcal{K}^{\prime}(u),u_{n}-u\rangle\leq 0.
\end{equation}
However, if we also consider that $\mathcal{K}^{\prime}$ is strictly monotone, we have
\begin{equation}\label{e3.2e}
\lim_{n\rightarrow\infty}\langle \mathcal{K}^{\prime}(u_{n})-\mathcal{K}^{\prime}(u),u_{n}-u\rangle=\lim_{n\rightarrow\infty}\langle \mathcal{K}^{\prime}(u),u_{n}-u\rangle=0.
\end{equation}
Therefore,
\begin{align}\label{e3.2f}
&\lim_{n\rightarrow\infty}\langle \mathcal{K}^{\prime}(u_n)-\mathcal{K}^{\prime}(u),u_n-u \rangle\nonumber \\
& = \lim_{n\rightarrow\infty}\left(\int_{\Omega}\left(\mathcal{M}_{\varrho}(u_n)|\nabla u_n|^{p(x)-2}\nabla u_n-\mathcal{M}_{\varrho}(u)|\nabla u|^{p(x)-2}\nabla u\right) \nabla(u_n-u) dx \right. \nonumber \\
& \left. + \int_{\Omega}\mu(x)\left(\mathcal{M}_{\varrho}(u_n)|\nabla u_n|^{q(x)-2}\nabla u_n-\mathcal{M}_{\varrho}(u)|\nabla u|^{q(x)-2}\nabla u\right) \nabla(u_n-u) dx\right)=0.
\end{align}
Recall the inequality \cite{lindqvist2019notes}
\begin{equation}\label{e3.2g}
(|\xi|^{h-2}\xi-|\eta|^{h-2}\eta)\cdot(\xi -\eta)>0, \quad \forall \xi ,\eta \in \mathbb{R}^{N},\, h\geq1.
\end{equation}
Then using $(M)$, (\ref{e3.2f}) and (\ref{e3.2g}) together imply that $\nabla u_n \to \nabla u$ and $u_n \to u$ in measure in $\Omega$. Then, by Riesz theorem, there are a subsequences, not relabelled, that
converges pointwise a.e. on $\Omega$ to $\nabla u$ and $u$, respectively \cite{royden2010real}. Hence, if we let
\begin{equation}\label{e3.2h}
\gamma_n(x):=\frac{|\nabla u_n|^{p(x)}}{p(x)}+\mu(x)\frac{|\nabla u_n|^{q(x)}}{q(x)},
\end{equation}
and
\begin{equation}\label{e3.2k}
\gamma(x):=\frac{|\nabla u|^{p(x)}}{p(x)}+\mu(x)\frac{|\nabla u|^{q(x)}}{q(x)},
\end{equation}
then $|\gamma_n(x)-\gamma(x)| \to 0$ a.e. on $\Omega$.\\
On the other hand, applying the Young's inequality it reads
\begin{align}\label{e3.2m}
\langle \mathcal{K}^{\prime}(u_n),u_n-u \rangle &=\mathcal{M}_{\varrho}(u_n)\int_{\Omega}\left(|\nabla u_n|^{p(x)-2}\nabla u_n+\mu(x)|\nabla u_n|^{q(x)-2}\nabla u_n\right) \nabla(u_n-u) dx  \nonumber \\
& \geq \mathcal{M}_{\varrho}(u_n)\int_{\Omega}\left(|\nabla u_n|^{p(x)}+\mu(x)|\nabla u_n|^{q(x)}\right) dx  \nonumber \\
& - \mathcal{M}_{\varrho}(u_n)\int_{\Omega}\left(|\nabla u_n|^{p(x)-1}|\nabla u|+\mu(x)|\nabla u_n|^{q(x)-1}|\nabla u|\right)dx \nonumber \\
& \geq \mathcal{M}_{\varrho}(u_n)\int_{\Omega}\left(\frac{|\nabla u_n|^{p(x)}}{p(x)}+\mu(x)\frac{|\nabla u_n|^{q(x)}}{q(x)}\right) dx  \nonumber \\
& - \mathcal{M}_{\varrho}(u_n)\int_{\Omega}\left\{\frac{p(x)-1}{p(x)}|\nabla u_n|^{p(x)}+\frac{1}{p(x)}|\nabla u|^{p(x)} \right. \nonumber \\
& \left. +\mu(x)\left(\frac{q(x)-1}{q(x)}|\nabla u_n|^{q(x)}+\frac{1}{q(x)}|\nabla u|^{q(x)} \right)\right\}dx \nonumber \\
& =\int_{\Omega}\mathcal{M}_{\varrho}(u_n)(\gamma_n(x)-\gamma(x))dx.
\end{align}
Then using (\ref{e3.2b}),(\ref{e3.2m}), Fataou's lemma, and Proposition \ref{Prop:2.7} together provides
\begin{equation}\label{e3.2n}
\lim_{n\rightarrow\infty}\int_{\Omega}\mathcal{M}_{\varrho}(u_n)(\gamma_n(x)-\gamma(x))dx=0.
\end{equation}
By the assumption $(M)$, we infer from (\ref{e3.2n}) that
\begin{equation}\label{e3.2p}
\lim_{n\rightarrow\infty}\int_{\Omega}\gamma_n(x)dx=\int_{\Omega}\gamma(x)dx.
\end{equation}

By the Vitali Convergence theorem \cite[Theorem 4.5.4]{bogachev2007measure}, (\ref{e3.2p}) means $|\nabla u_n| \to |\nabla u|$ and $\mu(x)^{\frac{1}{q(x)}}|\nabla u_n| \to \mu(x)^{\frac{1}{q(x)}}|\nabla u|$ in measure in $\Omega$  and the sequences $\left\{|\nabla u_n|^{p(x)}\right\}_{n}$ and $\left\{\mu(x)|\nabla u_n|^{q(x)}\right\}_{n}$ are uniformly integrable over $\Omega$, and hence, the functions family $\{\gamma_n\}_{n}=\left\{\frac{|\nabla u_n|^{p(x)}}{p(x)}+\mu(x)\frac{|\nabla u_n|^{q(x)}}{q(x)}\right\}_{n}$ is uniformly integrable over $\Omega$. For the rest, following the same arguments as developed in the proof of Lemma \ref{Lem:4.2}-$(ii)$ shows that the family $\left\{\frac{|\nabla (u_n-u)|^{p(x)}}{p(x)}+\mu(x)\frac{|\nabla (u_n-u)|^{q(x)}}{q(x)}\right\}_{n}$ is also uniformly integrable over $\Omega$.
Lastly, using the Vitali Convergence theorem once more, we obtain
\begin{equation}\label{e3.2s}
\lim_{n \to \infty}\int_{\Omega}\left(\frac{|\nabla(u_n-u)|^{p(x)}}{p(x)}+\mu(x)\frac{|\nabla(u_n-u)|^{q(x)}}{q(x)}\right)dx=0,
\end{equation}
which implies, by Proposition \ref{Prop:2.2a}, that $u_{n} \to u \in W_0^{1,\mathcal{H}}(\Omega)$.\\
$(v)$ Using $(M)$ and Remark \ref{Rem:2.1a}, it reads
\begin{align}\label{e3.21a}
\langle \mathcal{K}^{\prime}(u),u\rangle& = \mathcal{M}_{\varrho}(u)\int_{\Omega}(|\nabla u|^{p(x)}+\mu(x)|\nabla u|^{q(x)})dx \geq \frac{\kappa_{1}}{(q^+)^{\alpha_{1}}}\|u\|_{1,\mathcal{H},0}^{\alpha_{1}p^{-}}
\end{align}
or
\begin{align}\label{e3.21b}
&\frac{\langle \mathcal{K}^{\prime}(u),u\rangle }{\|u\|_{1,\mathcal{H},0}}\geq \frac{\kappa_{1}}{(q^+)^{\alpha_{1}}}\|u\|_{1,\mathcal{H},0}^{\alpha_{1}p^{-}-1},
\end{align}
which means that $\frac{\langle \mathcal{K}^{\prime}(u),u\rangle }{\|u\|_{1,\mathcal{H},0}} \to \infty$ as $\|u\|_{1,\mathcal{H},0}\to \infty$, so $\mathcal{K}^{\prime}$ is coercive.\\
Moreover, since $\mathcal{K}^{\prime}$ is also strictly monotone, $\mathcal{K}^{\prime}$ is an injection. According to Minty-Browder theorem \cite{zeidler2013nonlinear}, these two properties together implies that $\mathcal{K}^{\prime }$ is a surjection. Therefore, $\mathcal{K}^{\prime }$ has an inverse mapping $(\mathcal{K}^{\prime })^{-1}:W_0^{1,\mathcal{H}}(\Omega)^{*}\rightarrow W_0^{1,\mathcal{H}}(\Omega)$.\\
To show that $(\mathcal{K}^{\prime })^{-1}$ is continuous, let $(u_{n}^{\ast }),u^{\ast }\in W_0^{1,\mathcal{H}}(\Omega)^{*}$ with $
u_{n}^{\ast }\rightarrow u^{\ast }$, and let $(\mathcal{K}^{\prime })^{-1}(u_{n}^{\ast })=u_{n},(\mathcal{K}^{\prime})^{-1}(u^{\ast
})=u$. Then, $\mathcal{K}^{\prime} (u_{n})=u_{n}^{\ast }$ and $\mathcal{K}^{\prime}(u)=u^{\ast }$ which
means, by the coercivity of $\mathcal{K}^{\prime}$, that $( u_{n}) $ is bounded in $W_0^{1,\mathcal{H}}(\Omega)$. Thus, there exist $\hat{u}\in W_0^{1,\mathcal{H}}(\Omega)$ and a subsequence, not relabelled, $(u_{n})\subset W_0^{1,\mathcal{H}}(\Omega)$ such that $u_{n}\rightharpoonup \hat{u}$ in $W_0^{1,\mathcal{H}}(\Omega)$. However, by the uniqueness of the weak limit, $\hat{u}=u$ in $W_0^{1,\mathcal{H}}(\Omega)$. Additionally, since $u_{n}^{\ast}\rightarrow u^{\ast}$ in $W_0^{1,\mathcal{H}}(\Omega)^{*}$, it reads
\begin{equation}\label{e3.21c}
\lim_{n\rightarrow\infty}\langle \mathcal{K}^{\prime}(u_{n})-\mathcal{K}^{\prime}(u),u_{n}-u\rangle=\lim_{n\rightarrow\infty}\langle u_{n}^{\ast }-u^{\ast },u_{n}-u\rangle = 0.
\end{equation}
Considering that $\mathcal{K}^{\prime}$ is of type $(S_{+})$, we have $u_{n} \rightarrow u$ in $W_0^{1,\mathcal{H}}(\Omega)$. This concludes that $(\mathcal{K}^{\prime })^{-1}:W_0^{1,\mathcal{H}}(\Omega)^{*}\rightarrow W_0^{1,\mathcal{H}}(\Omega)$ is continuous.\\

$(vi)$ Lastly, we shall show that $\mathcal{K}$ is sequentially weakly lower semi-continuous.\\
Since $\mathcal{K}^{\prime}$ is continuously G\^{a}teaux differentiable, it is enough to show that $\mathcal{K}$ is convex. To show this, we adopt the approach used in \cite[Lemma 3.4]{avci2024existence}, and for the sake of the completeness we provide the some details. Since $\mathcal{K}(u)=\widehat{\mathcal{M}}_{\varrho}(u)$, we first show that $\widehat{\mathcal{M}}$ is convex and increasing over $(0,\infty)$. By Proposition \ref{Prop:2.1}, $\widehat{\mathcal{M}}$ is convex if
\begin{equation*}
\widehat{\mathcal{M}}((1-\epsilon)t+\epsilon s)< (1-\epsilon)\beta_1+\epsilon \beta_2, \,\,\,\, 0<\epsilon<1,
\end{equation*}
whenever $\widehat{\mathcal{M}}(t)<\beta_1$ and $\widehat{\mathcal{M}}(s)<\beta_2$, for all $t,s,\beta_1,\beta_2 \in (0,\infty)$. Thus, applying $(M)$, it reads
\begin{align}\label{e3.21d}
\widehat{\mathcal{M}}(t)\leq \frac{\kappa_2}{\alpha_2}t^{\alpha_{2}}< \frac{2^{\alpha_{2}-1}\kappa_2}{\alpha_2}t^{\alpha_{2}}:=\beta_1 \text{ and }
\widehat{\mathcal{M}}(s)\leq \frac{\kappa_2}{\alpha_2}s^{\alpha_{2}}< \frac{2^{\alpha_{2}-1}\kappa_2}{\alpha_2}s^{\alpha_{2}}:=\beta_2.
\end{align}
Therefore,
\begin{align}\label{e3.21e}
\widehat{\mathcal{M}}((1-\epsilon)t+\epsilon s)\leq \frac{\kappa_2}{\alpha_2}[(1-\epsilon)t+\epsilon s]^{\alpha_{2}}< (1-\epsilon)\frac{2^{\alpha_{2}-1}\kappa_2}{\alpha_2}t^{\alpha_{2}}+\epsilon\frac{2^{\alpha_{2}-1}\kappa_2}{\alpha_2}s^{\alpha_{2}},
\end{align}
which shows that $\widehat{\mathcal{M}}$ is convex. To show that $\widehat{\mathcal{M}}$ is increasing over $(0,\infty)$, one can just apply the Fundamental Theorem of Calculus and the condition $(M)$. Note also that, since $\widehat{\mathcal{M}}$ is convex over $(0,\infty)$, it is continuous on $(0,\infty)$. Putting all these together; since $\widehat{\mathcal{M}}$ is convex and increasing on $(0,\infty)$, and the functional $\varrho_{\mathcal{H}}$ is convex on $W_0^{1,\mathcal{H}}(\Omega)$, as the composition of these two maps, $\mathcal{K}$ is also convex on $W_0^{1,\mathcal{H}}(\Omega)$.
\end{proof}

\begin{lemma}\label{Lem:4.3}
Assume that $(f_{1})$ holds. Then, $\mathcal{J}$ is a continuously G\^{a}teaux differentiable functional whose G\^{a}teaux derivative $\mathcal{J^{\prime}}:W_0^{1,\mathcal{H}}(\Omega)\rightarrow W_0^{1,\mathcal{H}}(\Omega)^{*}$ given by
\begin{equation}\label{e3.4m}
\langle \mathcal{J^{\prime}}(u),\varphi\rangle=\int_{\Omega}f(x,u)\varphi dx, \quad \forall u, \varphi \in W_0^{1,\mathcal{H}}(\Omega)
\end{equation}
is compact.
\end{lemma}
\begin{proof} To show that $\mathcal{J}$ is continuously G\^{a}teaux differentiable, one can argue similarly to Lemma \ref{Lem:4.2}, part $(ii)$. Hence we omit that part and continue with showing that  $\mathcal{J^{\prime}}$ is compact.\\
Define the operator $\mathcal{A}_f:W_0^{1,\mathcal{H}}(\Omega)\rightarrow L^{s^{\prime}(x)}(\Omega)$ by
\begin{equation} \label{e3.1a1}
\mathcal{A}_f(u):=f(x,u).
\end{equation}
With this characterization, the operator $\mathcal{A}_f$ is $L^{s^{\prime}(x)}(\Omega)$-norm bounded. Indeed, let $\|u\|_{1,\mathcal{H},0}\leq1$. Using the embeddings $L^{\mathcal{H}}(\Omega)\hookrightarrow L^{s(x)}(\Omega)$, and $W_0^{1,\mathcal{H}}(\Omega)\hookrightarrow L^{\mathcal{H}}(\Omega)$, it reads
\begin{align}\label{e3.1a2}
\int_{\Omega} |\mathcal{A}_f(u)|^{s^\prime(x)} dx&=\int_{\Omega} |f(x,u)|^{s^\prime(x)} dx\leq \int_{\Omega} \bigg|\bar{c}_{1}+ \bar{c}_{2}|u_{n}|^{s(x)-1}\bigg|^{s^\prime(x)}dx  \nonumber \\
&\leq c_1\int_{\Omega} |u|^{s(x)}dx+c_2|\Omega|\nonumber\\
&\leq c_3\int_{\Omega} \left(|u|^{s(x)}+|\nabla u|^{p(x)}+\mu(x)|\nabla u|^{q(x)}\right)dx+c_2|\Omega| \nonumber\\
&\leq c_4\|u\|_{1,\mathcal{H},0}+c_2|\Omega|,
\end{align}
where $c_1,c_2,c_3,c_4>0$ are some real parameters whose values are independent of $u$.\\
Next, let $u_n \to u$ in $W_0^{1,\mathcal{H}}(\Omega)$, and hence $u_n \to u$ in $L^{s(x)}(\Omega)$. Due to the standard arguments, there exist a subsequence $(u_n)$, not relabelled, and a function $\omega$ in $L^{s(x)}(\Omega)$ satisfying
\begin{itemize}
\item $u_n(x) \to u(x)$ a.e. in $\Omega$,
\item $|u_n(x)|\leq \omega(x)$ a.e. in $\Omega$ and for all $n$.
\end{itemize}
By $(f_1)$, we have
\begin{equation} \label{e3.1a3}
f(x,u_n(x))\to f(x,u(x)) \text{  a.e. in } \Omega,
\end{equation}
and
\begin{align} \label{e3.1a4}
|f(x,u_n(x))|& \leq \bar{c}_{1}+\bar{c}_{2}|\omega(x)|^{s(x)-1}.
\end{align}
Using the Young's inequality, it reads
\begin{equation}\label{e3.1a5}
|f(x,u_n(x))| \leq \bar{c}_{1}+\frac{(\bar{c}_{2})^{s_{+}}}{s^-}+\frac{s^+-1}{s^-}|\omega(x)|^{s(x)}.
\end{equation}
Then by the embeddings $L^{\mathcal{H}}(\Omega)\hookrightarrow L^{s(x)}(\Omega)$ and $W_0^{1,\mathcal{H}}(\Omega)\hookrightarrow L^{\mathcal{H}}(\Omega)$, the right-hand side of (\ref{e3.1a4}) is integrable. Therefore, by (\ref{e3.1a3}) and the Lebesgue dominated convergence theorem, we obtain
\begin{equation}\label{e3.1a7}
f(x,u_{n}) \rightarrow f(x,u)\,\, \text {in}\,\, L^{1}(\Omega),
\end{equation}
and hence
\begin{align}\label{e3.1a8}
\lim_{n \to \infty}\int_{\Omega} |\mathcal{A}_f(u_n)-\mathcal{A}_f(u)|^{s^\prime(x)} dx=0,
\end{align}
which, by Proposition \ref{Prop:2.2}, implies  that $\mathcal{A}_f$ is continuous in $L^{s^{\prime}(x)}(\Omega)$.\\
Next, consider the compact imbedding operator $i: W_0^{1,\mathcal{H}}(\Omega) \to L^{s(x)}(\Omega)$. Then the adjoint operator $i^*$ of $i$, given by $i^*:L^{s^{\prime}(x)}(\Omega) \to W_0^{1,\mathcal{H}}(\Omega)^*$  is also compact. Therefore, if we set $\mathcal{J^{\prime}}= i^*\circ \mathcal{A}_f$, then $\mathcal{J^{\prime}}:W_0^{1,\mathcal{H}}(\Omega)\rightarrow W_0^{1,\mathcal{H}}(\Omega)^{*}$ is compact.
\end{proof}

We're ready to present the proof of the main result.

\begin{proof}(\textbf{Proof of Theorem \ref{Thm:4.1}})\\
The required regularity conditions for $\mathcal{K}$ and $\mathcal{J}$ are proved in Lemmas \ref{Lem:4.2} and \ref{Lem:4.3}.\\
Moreover, it is clear that
\begin{equation}\label{e4.5}
\inf_{u \in W_0^{1,\mathcal{H}}}\mathcal{K}(u)=\mathcal{K}(0)=\mathcal{J}(0)=0.
\end{equation}
Note that if we define
\begin{equation}\label{e4.6}
\delta(x)=\sup\{\delta>0: B(x,\delta)\subseteq \Omega\}, \quad \forall x \in \Omega,
\end{equation}
and consider that $\Omega$ is open and connected in $\mathbb{R}^N$, it can easily be shown that there exists $x_0 \in \Omega$ such that $B(x_0,R)\subseteq \Omega$ with $R=\sup_{x \in \Omega}\delta(x)$.\\
First, we define the cut-off function  $\bar{u} \in W_{0}^{1,\mathcal{H}}(\Omega)$ by the formula
\begin{equation}\label{e4.7}
\bar{u}(x)=
\left\{
\begin{array}{ll}
0,&\text{if}\ x \in \Omega \setminus B(x_0,R),\\
\delta,&\text{if}\ x \in B(x_0,R/2),\\
\frac{2\delta}{R}(R-|x-x_0|),&\text{if}\ x \in B(x_0,R)\setminus B(x_0,R/2):=\hat{B}.
\end{array}
\right.
\end{equation}
Then using $(M)$ we have
\begin{align}\label{e4.8a}
\mathcal{K}(\bar{u})&=\int^{\varrho_{\mathcal{H}}(\bar{u})}_{0}\mathcal{M}(s)ds \geq \frac{\kappa_{1}}{\alpha_{1} (q^+)^{\alpha_{1}}}\left(\int_{\hat{B}}\left(|\nabla \bar{u}|^{p(x)}+\mu(x)|\nabla \bar{u}|^{q(x)}\right) dx\right)^{\alpha_{1}} \nonumber\\
& \geq \frac{\kappa_{1}}{\alpha_{1} (q^+)^{\alpha_{1}}}\left(\frac{2\delta}{R}\right)^{\alpha_{1}(p^- \wedge p^+)}\omega_N^{\alpha_{1}}\left(R^{N}-\left(\frac{R}{2}\right)^{N}\right)^{\alpha_{1}}\nonumber\\
& = \frac{\kappa_{1}}{\alpha_{1} (q^+)^{\alpha_{1}}} \omega_N^{\alpha_{1}} R^{N\alpha_{1}}(2^{N}-1)^{\alpha_{1}}2^{-N\alpha_{1}}\left(\frac{2\delta}{R}\right)^{\alpha_{1}(p^- \wedge p^+)},
\end{align}
and similarly
\begin{align}\label{e4.8b}
\mathcal{K}(\bar{u})& \leq \frac{\kappa_{2}}{\alpha_{2} (p^-)^{\alpha_{2}}}\left(\int_{\hat{B}}\left(|\nabla \bar{u}|^{p(x)}+\mu(x)|\nabla \bar{u}|^{q(x)}\right) dx\right)^{\alpha_{2}} \nonumber\\
& \leq \frac{\kappa_{2}(1+|\mu|_{\infty})^{\alpha_{2}}}{\alpha_{2} (p^-)^{\alpha_{2}}} \omega_N^{\alpha_{2}}\left(R^{N}-\left(\frac{R}{2}\right)^{N}\right)^{\alpha_{2}}\left(\frac{2\delta}{R}\right)^{\alpha_{2}(p^- \vee q^+)}.
\end{align}
where $\omega_{N}:=\frac{\pi^{N/2}}{N/2\Gamma(N/2)}$ is the volume of the unit ball in $\mathbb{R}^N$, where $\Gamma$ is the Gamma function.\\
By $(f_3)$, it reads
\begin{align}\label{e4.10}
\mathcal{J}(\bar{u})&=\int_{B(x_0,R/2)} F(x,\bar{u})dx\geq \inf_{x \in \Omega}F(x,\delta)\omega_N\left(\frac{R}{2}\right)^{N}.
\end{align}
Therefore
\begin{align}\label{e4.11}
\frac{\mathcal{J}(\bar{u})}{\mathcal{K}(\bar{u})}&\geq \frac{2^{N(\alpha_{2}-1)}\alpha_{2} (p^-)^{\alpha_{2}}\inf_{x \in \Omega}F(x,\delta)}{\kappa_{2}(1+|\mu|_{\infty})^{\alpha_{2}}\omega_N^{\alpha_{2}-1}(2^{N}-1)^{\alpha_{2}}R^{N(\alpha_{2}-1)}\left(\frac{2\delta}{R}\right)^{\alpha_{2}(p^- \vee q^+)}}.
\end{align}
Note that for any $u \in \mathcal{K}^{-1}((-\infty, r])$, by $(M)$ we have
\begin{align}\label{e4.11a}
\frac{\kappa_{1}}{\alpha_{1}} \varrho_{\mathcal{H}}(u)^{\alpha_{1}}\leq \mathcal{K}(u)\leq r.
\end{align}
Thus, by the involved embeddings it reads
\begin{align}\label{e4.12}
\mathcal{J}(u)&\leq \int_{\Omega} \bar{c}_{1}|u|dx + \int_{\Omega} \bar{c}_{2}|u|^{s(x)}dx\leq \bar{c}_{1}c_{1}(\mathcal{H})\|u\|_{1,\mathcal{H},0} + \bar{c}_{2}c_{2}(\mathcal{H})^{s^+}\|u\|_{1,\mathcal{H},0}^{s^+}\nonumber\\
& \leq \bar{c}_{1}c_{\mathcal{H}}(q^+)^{\frac{1}{p^-}}\left(\frac{\alpha_{1}}{\kappa_{1}}\right)^{\frac{1}{\alpha_1(p^- \wedge q^+)}}  r^{\frac{1}{\alpha_1(p^- \wedge q^+)}} + \bar{c}_{2}c_{\mathcal{H}}^{s^+}(q^+)^{\frac{s^+}{p^-}} \left(\frac{\alpha_{1}}{\kappa_{1}}\right)^{\frac{s^- \wedge s^+}{\alpha_1(p^- \wedge q^+)}}  r^{\frac{s^- \wedge s^+}{\alpha_1(p^- \wedge q^+)}},
\end{align}
and hence,
\begin{align}\label{e4.12a}
&\frac{1}{r}\mathcal{J}(u)\nonumber\\
&\leq \frac{1}{r}\biggl\{\bar{c}_{1}c_{\mathcal{H}}(q^+)^{\frac{1}{p^-}}\left(\frac{\alpha_{1}}{\kappa_{1}}\right)^{\frac{1}{\alpha_1(p^- \wedge q^+)}}  r^{\frac{1}{\alpha_1(p^- \wedge q^+)}} + \bar{c}_{2}c_{\mathcal{H}}^{s^+}(q^+)^{\frac{s^+}{p^-}} \left(\frac{\alpha_{1}}{\kappa_{1}}\right)^{\frac{s^- \wedge s^+}{\alpha_1(p^- \wedge q^+)}}  r^{\frac{s^- \wedge s^+}{\alpha_1(p^- \wedge q^+)}}\biggl\},
\end{align}
where $c_{\mathcal{H}}:=\max\{c_{1}(\mathcal{H}),c_{2}(\mathcal{H})\}$ and $c_{1}(\mathcal{H}),c_{2}(\mathcal{H})$ are the best embedding constants.
Then, using $(f_4)$ and (\ref{e4.11}) provides
\begin{align}\label{e4.13}
\frac{1}{r}\sup_{\mathcal{K}(u)\leq r}\mathcal{J}(u)< \frac{\mathcal{J}(\bar{u})}{\mathcal{K}(\bar{u})}.
\end{align}
Thus, the condition $(a1)$ of Lemma \ref{Lem:3.1} is verified.\\
Next, we show that the condition $(a2)$ of Lemma \ref{Lem:3.1} is satisfied; that is, the functional $\mathcal{I}_{\lambda}(\cdot):=\mathcal{K}(\cdot)-\lambda \mathcal{J}(\cdot)$ is coercive.\\
Using the embedding $W_0^{1,\mathcal{H}}(\Omega)\hookrightarrow L^{r(x)}(\Omega)$, and Lemma \ref{Lem:4.2}-$(i)$, it reads
\begin{equation}\label{e4.14}
\mathcal{I}_{\lambda}(u)\geq \frac{\kappa_{1}}{\alpha_{1} (q^+)^{\alpha_{1}}}\|u\|_{1,\mathcal{H},0}^{\alpha_{1} p^{-}}-\lambda \bar{c}_{3}(c_{\mathcal{H}}^{r^+}\|u\|_{1,\mathcal{H},0}^{r^{+}}+|\Omega|),
\end{equation}
which implies that for any $\lambda>0$, $\mathcal{I}_{\lambda}$ is coercive.\\
In conclusion, by Lemma \ref{Lem:3.1}, for any $\lambda \in \lambda_{r,\delta}\subseteq \left(\frac{\mathcal{K}(\bar{u})}{\mathcal{J}(\bar{u})},\frac{r}{\sup_{\mathcal{K}(u)\leq r}\mathcal{J}(u)} \right)$, the problem (\ref{e1.1a}) admits at least three distinct weak solutions.
\end{proof}

\section{Conclusions}\label{sec4}
In this work, we have investigated a class of double-phase variable-exponent Kirchhoff problems and established the existence of at least three distinct weak solutions. The problem framework extends classical Kirchhoff-type equations by incorporating a double-phase operator with variable growth, thereby capturing anisotropic and heterogeneous diffusion phenomena within a unified variational setting. A key contribution of this paper lies in the rigorous analysis of the associated energy functional, for which we have proven essential regularity properties such as $C^1$-smoothness, the $(S_+)$-condition, and sequential weak lower semicontinuity. These results enable the application of a three critical point theorem due to Bonanno and Marano, ensuring the existence of multiple solutions under suitable assumptions. As a novel tool in the analysis, we introduce an $n$-dimensional vector inequality (Proposition \ref{Prop:2.2bc}), which can be utilized to handle the technical challenges posed by variable-exponent, nonstandard growth functionals. This auxiliary result plays a fundamental role in proving the differentiability and compactness properties of the functional. Our findings extend the theoretical framework for Kirchhoff-type problems with a double-phase structure. Future work may explore further generalizations, including nonlocal operators, critical growth terms, or imposing singular perturbations in the same setting.

\section*{Declarations}
\section*{Data Availability}
No new data were created or analyzed in this study.
\section*{Funding}
This research received no external funding.
\section*{ORCID}
https://orcid.org/0000-0002-6001-627X

\bibliographystyle{tfnlm}
\bibliography{references}

\end{document}